\documentclass[12pt,a4paper, notitlepage]{article}
\usepackage[a4paper] {geometry}
\usepackage{amsmath,amsfonts,amssymb,mathrsfs,amsthm}
\usepackage{multirow}
\usepackage[english]{babel}
\usepackage[all]{xy}
\usepackage{tikz}
\numberwithin{equation}{section}
\numberwithin{figure}{section}
\newtheorem{thm}{Theorem}[section]
\newtheorem{prop}[thm]{Proposition}
\newtheorem{lem}[thm]{Lemma}
\newtheorem{coro}[thm]{Corollary}
\theoremstyle{definition}
\newtheorem{defx}[thm]{Definition}
\newtheorem{rem}[thm]{Remark}

\newcommand{\R}{I\!\!R}

\DeclareMathSizes{12}{13}{10}{10}

\title{Groupoid approach to the dynamical system of commutative von Neumann Algebras}
\author{{\emph{N. O. Okeke \;  Email: nickokeke@dui.edu.ng, }}
\\ Physical and Mathematical Sciences, Dominican University, Ibadan \\
\emph{M. E. Egwe \;  murphy.egwe@ui.edu.ng}
\\ Department of Mathematics, University of Ibadan, Ibadan, Nigeria}

\date{\vspace{-5ex}}

\begin{document}
\maketitle

\begin{abstract}
The automorphism group $Aut(X,\mu)$ of a compact, complete metric space $X$ with a Radon measure $\mu$ is a subgroup of $\mathcal{U}(L^2(X,\mu))$-the unitary group of operators on $L^2(X,\mu)$. The $Aut(X,\mu)$-action on the generalized space $\mathcal{M}(X)$ is a proper action. Hence, there exists a slice at each point of the generalized space $\mathcal{M}(X)$. Measure Groupoid (virtual group) is subsequently employed to analyze the resulting dynamical system as that of the ergodic action of the commutative algebra (a lattice) $C(X)$ on the generalized space $\mathcal{M}(X)$ which is represented on a commutative von Neumann algebra.
\end{abstract}
\let\thefootnote\relax
\let\thefootnote\relax

\section*{Introduction}
Following the conception of dynamical symmetries in physical systems by F. Strocchi \cite{Strocchi2008} as the invariance of the system under the transformations of coordinates or parameters identifying its configurations, or the commutation of action of the symmetry (group) transformations with time evolution of the system, we specify the phase space as constituted by the parametrized geometric space $\{\mathcal{H}_x\}_{x \in X}$, with $X$ a completely regular compact metric space, the time evolution of the system is given by the continuous transformation $\varphi : X \to X$ of the compact space. The transformation $\varphi : x \to \varphi(x)$ defines a symmetry, according to \cite{Strocchi2008}, if \\
(a) it induces an invertible mapping of configurations $\varphi_* : \mathcal{H}_x \to \varphi_*(\mathcal{H}_x) \equiv \mathcal{H}_{\varphi(x)}$; \\
(b) it leaves the dynamical behaviour of the system invariant \[\alpha^t\varphi(\mathcal{H}_x) = \alpha^t\mathcal{H}_{\varphi(x)} \equiv \mathcal{H}_{\varphi(x)(t)} = \mathcal{H}_{\varphi(x(t))} = \varphi(\alpha^t\mathcal{H}_x).\]
These have to do with the concept of ergodicity; relating average over time as given by the iteration of transformations, to average over space given by invariance of measures, and the measurable functions representing the observables. In the case of the action of a compact Lie group on a smooth manifold, the fixed condition of invariance of Haar measure gave rise to the cohomological equation, the solution of which afforded the harmonic analysis of the discrete dynamical system \cite{OkekeEgwe2017}. When the smooth manifold is replaced by a generalized space of Radon measures on $X$, what kind of orbits will this give? What is the nature of the symmetry preserved in the orbits? 

Our interest in this paper is to address these questions, thereby extend the results of \cite{OkekeEgwe2018} in cohomogeneity one $G$-space groupoid analysis to (measure) groupoid actions on the generalized space $\mathcal{M}(X)$ which will be shown to be a commutative von Neumann algebra. Both Peter Hahn \cite{Hahn78} and Arlan Ramsay \cite{Ramsay71} have used measure groupoid to understand Mackey's virtual groups and measure theory that usually results from the consideration. Our focus is more on bringing out the implication of the slice theorem on the measure groupoid approach to  commutative von Neumann algebras.
The proper actions connect directly to slice theorem used in the cohomogeneity-one $G$-space analysis. Therefore the main results of the paper will be stated in view of this concern.

\section{Algebraic and Topological Structures}
Given the commutative algebra $C(X)$, we make use of the prime/maximal ideals which, according to Atiyah and Macdonald \cite{AtiyahDonald69}, play a central role in commutative algebras because of their generalization of the idea of prime numbers in number theory and points in geometry. Subsequently, the restriction of attention to neighbourhood of a point is subsequently captured by the process of localization at a prime ideal.
From the characterization of commutative rings or algebra in \cite{AtiyahDonald69}, which we recommend for basic understanding of a commutative ring or algebra, we highlight the following that connect them immediately to groupoid.

The first is that by definition of an $R$-module $M$ of a ring $R$, the map $R \to End(M)$ is a representation of $R$ on the space of linear transformations of $M$. The second is that given a ring $R$ (or an algebra $\mathcal{A}$), by definition, the units of the ring $R$ form a group $G(1)$ with an action $G(1) \times M \to M$ on every $R$-module $M$. The action gives us an action groupoid or a transformation groupoid $\mathcal{G} = G(1) \ltimes M$ over each module $M$. Thus, $\mathcal{G} = G(1) \times M \rightrightarrows M$ is an action groupoid. We will now highlight the algebraic and topological structures of the commutative algebra $C(X)$ which will help us understand the structure $X$ from the action groupoid trivialized on it.

According to Gillman and Jerison \cite{Gillman1960}, given a compact space $X$, every real valued function on $X$ is continuous if $X$ is discrete. So $\R^X = C(X)$. Conversely, whenever $\R^X = C(X)$ then the characteristic function of every set in $X$ is continuous, showing that $X$ is discrete. This equivalence between the continuity of a real valued functions on $X$ and the discreteness of $X$ is the foundation for using the algebraic variety (or the zero-sets) of $C(X)$ to characterize a regular compact space $X$ in \cite{Gillman1960}.
\begin{defx}\cite{Gillman1960}
A space $X$ is said to be completely regular if it is Hausdorff and whenever $F$ is a closed set and $x \in F^c$, there exists a function $f \in C(X)$ such that $f(x) = 1$ and $f(F) = \{0\}$.
\end{defx}
\begin{defx}
The closed sets $f^{-1}(0) = \ker(f) = \{x \in X : f(x) = 0\}$ are called the zero sets of $C(X)$. On the other hand, the open sets $pos f = \{x : f(x) > 0\}$ and $neg f = \{x : f(x) < 0\} = pos (-f)$ are called cozero sets since they are complements of zero sets. Alternatively, every cozero set is of the form $X - Z(f) = pos |f|$.
\end{defx}
The definition shows how the regularity of the space $X$ is connected to the zero-sets of $C(X)$. Hence, we characterize the zero-sets by using subsets of the form $$f^{-1}(r) = \{x \in X : f(x) = r\}  \equiv  f^{-1}(0) = \ker(f) = \{x \in X : f(x) = 0\}.$$
The importance of these subsets lies in the fact that they are closed subsets. The zero-sets of $C(X)$ define the essential properties of the space $X$, which are those properties invariant under translation of zero-sets or transformation of $X$. We refer the reader to \cite{Gillman1960} for a complete understanding of how the zero-sets define the algebraic and topological structures of the compact regular space $X$. Two maps are defined from the characterization of zero sets of $C(X)$ in $X$ which contain the essentials for our goal in this paper.

The first map is the zero-set map $Z : C(X) \to Z(X)$ defined by $f \mapsto Z(f)$. It maps a function $f \in C(X)$ to its zero-set or kernel in $X$.
The image of $C(X)$ under the zero-set map $Z$ is the family $Z(X)$ of all zero-sets which is a base for closed sets. The weak topology on $C(X)$ relates to the topology defined on $X$ by $Z(X)$ as the base of closed sets. Cf. \cite{Gillman1960}.

The following properties follow from this definition. (a) $Z(f) = Z(|f|) = Z(f^n), \; \forall \; n \in \mathbb{N}$; (b) $Z(0) = X, Z(1) = \emptyset$; (c) $Z(fg) = Z(f) \cup Z(g)$; (d) $Z(f^2 + g^2) = Z(|f| + |g|) = Z(f) \cap Z(g)$. The following also characterize zero sets: (1) Every zero set is a $G_\delta$ (a countable intersection of open sets) since $Z(f) = \underset{n \in \mathbb{N}}\bigcap \{x \in X : |f(x)| < \frac{1}{n} \}$. (2) When the space $X$ is normal every closed $G_\delta$ is a zero-set. (3) When $X$ is a metric space, every closed set is a zero-set consisting of all points whose distance from it is zero. (4) Every set of the form $\{x : f(x) \geq 0\}$ is a zero-set, since $\{x : f(x) \geq 0\} = Z(f\wedge 0) = Z(f-|f|)$ (supremum) and $\{x : f(x) \leq 0\} = Z(f\vee 0) = Z(f + |f|)$ (infimum).

The second map is defined based on the correspondence between maximal ideals $\mathfrak{m}_x = \{f \in C(X) : x \in \ker f\}$ of $C(X)$, which are kernels of the surjective ring homomorphism $x : C(X) \to \R$ defined by $x(f) = f(x)$, and the points of $X$. This is given in \cite{AtiyahDonald69} as the homeomorphism $\mathfrak{z} : X \to C(X), x \mapsto m_x$ of $X$ onto $C(X)$. The homeomorphism identifies $X$ with the algebraic closed points of $C(X)$ under the Zariski topology. Because the topological and the algebraic structures of $C(X)$ are related to the convergence of $z$-filters in $z$-ultrafilters of $X$, we have the following result on the coincidence of the two topologies.
\begin{prop}
The $C(X)$-topology on $X$ coincides with the topology induced on $X$ by the Zariski topology on $C(X)$.
\end{prop}
\begin{proof}
Given the complete regularity of $X$ which ensures the convergence of $z$-filters in the space, the zero-sets $Z(X)$ of $X$ form a base for the closed sets. The closure of a subset $S \subset X$ is the set of all cluster points of the $z$-filter $\mathcal{F}$ of all zero-sets containing $S$. So, the $z$-filter $\mathcal{F}$ converges to the limit $x$ if every neighbourhood of $x$ contains a member of $\mathcal{F}$. Thus, with the $z$-ultrafilter $Z[\mathfrak{m}_x]$ converging to $x$ for each $x \in X$, we have the composition of the two bijections $Z\circ \mathfrak{z} : X \to X$ to be an identity on $X$.
\end{proof}

\section{Action of the Borel Group}
In line with our view of the commutative algebra $C(X)$ as a bundle of action groupoids based on its maximal ideals, we can consider the map $\mathfrak{z} : X \to C(X)$ as a section of a bundle and the map $Z : C(X) \to X$ as a projection from the fibres which are action groupoids or the maximal ideals $\mathfrak{m}_x$ with the action of $G(1)$. We will therefore consider $G(1)$ as a subgroup of the automorphism group of the $z$-ultrafilter $Z[\mathfrak{m}_x]$ determined by the fibres of the Lie groupoid defined by the action groupoid bundle over $X$. We will work out the detail of this structure in the sequel.

From the characterization of the group of units $G(1)$ above, it follows that $Z(g) = \emptyset$ for every $g \in G(1)$. Thus, $G(1)$ can be made to define various actions on the maximal ideal $\mathfrak{m}_x$ which reflect on the $z$-ultrafilter $Z[\mathfrak{m}_x]$ corresponding to the maximal ideal, converging at each $x \in X$. This is considered an affine $G(1)$-action on the module $\mathfrak{m}_x$, since the whole space $X$ with $\R$-action is locally convex. We define the following actions. \\
(1) First, given $f \in \mathfrak{m}_x$ and $g \in G(1)$, $Z(f)\cap Z(g) = \emptyset$, then $|f| + |g|$ is a unit since it has empty zero set; hence $Z(|f| + |g|) = \emptyset$. We define an action $G(1) \times \mathfrak{m}_x \to \mathfrak{m}_x$ by \[ (g,f) \mapsto h = \frac{|f|}{|f| + |g|}, \; \text{or } = h =  |f|\cdot(|f| + |g|)^{-1}. \] Then $h$ is defined by $h[Z(f)] = \{0\}$ and $h[X-Z(f)] = (0,1)\}$; and in general, $h : X \to [0,1)$. This is an action of $G(1)$ on the maximal ideal $\mathfrak{m}_x$ which preserves the zero-sets $Z(f) = Z(h)$ and the cozero-sets $X - Z(f) = X - Z(h)$ for every $f \in \mathfrak{m}_x$. Thus, the $G(1)$-action preserves the structures of $C(X)$ by preserving the open and closed sets in $X$ associated with the algebra $C(X)$. \\
(2) Second, given $f \in \mathfrak{m}_x$ and $g \in G(1)$, define a multiplicative action $G(1) \times \mathfrak{m}_x \to \mathfrak{m}_x$ defined by $\tau_g : f \mapsto gf$ or $f \mapsto fg^{-1}$, which translates the zero-set $Z(f) \mapsto Z(gf)$ or $Z(f) \mapsto Z(fg^{-1})$ and cozero sets $pos|f| \mapsto pos|gf|$ (or $pos|fg^{-1}|$ resp.) of $f \in \mathfrak{m}_x$. The action scales the functions in $\mathfrak{m}_x$.

The combination of these actions gives the affine $G(1)$-action on each maximal ideal $\mathfrak{m}_x$, which preserve the zero-sets and cozero-sets associated to the maximal ideals. We can therefore consider $G(1)$ to be a sort of transformation group of elements of the base (open/closed sets) of the topology of $X$ which are measurable or Borel sets.

Furthermore, since the $G(1)$-actions preserve the topological structure induced on $X$ by $C(X)$, it can also be considered to represent a topological action of $(0,1]$ (or $\R^*$) on the topological space $X$, and subsequently on the family $\mathfrak{F} =\{X - Z(f)\}_{f \in \mathfrak{m}_x} = \{U_f : f \in \mathfrak{m}_x \}$ generating the Borel structure $(X,\mathcal{B})$. The family $\mathfrak{F} = \{U_f : f \in \mathfrak{m}_x \}$ of open sets corresponds to the fibre $\mathfrak{m}_x$ at each $x \in X$. The continuity of the $G(1)$-actions follows from their preservation of the topological structure of $X$.

The Borel structure induced by the topology is subsequently preserved. Hence, given the topology on $X$ associated to the weak topology on $C(X)$, the $(X, \mathcal{B})$ is the Borel space associated with the topology on $X$ if $\mathcal{B}$ is the smallest $\sigma$-algebra of subsets of $X$ with respect to which all real-valued continuous functions are measurable. $\mathcal{B}$ is generated by sets of the form $f^{-1}(C)$, where $C$ is any closed subset of the real line and $f \in C(X)$. Within the background of translations of the zero-sets by elements of $G(1)$, we can define a Borel transformation of these open/closed base as follows.
\begin{defx}\cite{Varadarajan62}
An automorphism of a Borel space $(X,\mathcal{B})$ is a one-to-one map $\phi$ of $X$ onto itself such that $\phi^{-1}(A) \in \mathcal{B}$ if and only if $A \in \mathcal{B}$.
\end{defx}
Thus, given the Borel space $(X, \mathcal{B})$, a map $T : \mathfrak{F} \to \mathfrak{F}$ of the generating family of the Borel $\sigma$-algebra $\mathfrak{F}$, which is a subalgebra of $\mathcal{B}$ onto itself, is an automorphism of the neighbourhood base $(\mathfrak{F})$ of the topology of $X$ if it is one-to-one and $T^{-1}(A) \in \mathcal{B}$ whenever $A \in \mathcal{B}$.
\begin{rem}
The family $\mathfrak{B} = \{U_f : f \in \mathfrak{m}_x, x \in X \}$ of open sets separates the points of $X$. Thus, it implies that for any pair $x,y \in X$, there is $U_f \in \mathfrak{F}$ such that $x \in U_f$ and $y \not \in U_f$. $X$ is therefore separated.
Note also that the automorphism defined above is within a maximal or an ultrafilter of $Z(X)$ which ensures the meeting of images and pre-images.
\end{rem}
\begin{prop}
Given that $\mathcal{B}$ is the Borel structure generated on $X$ by the families $\mathfrak{F} = \{U_f : f \in \mathfrak{m}_x, x \in X \}$. The map $\varphi : g \mapsto T_g$ is a representation of $G(1)$ on the group of automorphisms $Aut(\mathfrak{F}_x)$ of the the generating family of the Borel $\sigma$-algebra $\mathcal{B}$, whereby $T_g : U_f \to U_{gf}$.
\end{prop}
\begin{proof}
Recall that the families $\mathfrak{F} = \{U_f : f \in \mathfrak{m}_x, x \in X \}$ constitute a basis for the topology on $X$. By definition, the open sets $U_f$ are cozero sets of $f \in \mathfrak{m}_x$ complement to the zero set $Z(f)$, for each $f \in \mathfrak{m}_x$-a $C(X)$-module. Thus, since $\mathcal{B}$ is the Borel structure associated with the topology on $X$ and $G(1)$ is the group of units in $C(X)$, the map, $\varphi : G(1) \to Aut(\mathcal{B})$ is a representation associated to the fibre action; when restricted to the generating subfamily $\mathfrak{F}$, corresponding to $Z[\mathfrak{m}_x], x \in X$, it is defined by $\varphi(g) = T_g$ such that $T_g(U_f) = U_{gf}$ and satisfies the properties of a group action:\\
(i) $T_{g_2g_1}(U_f) = T_{g_2}(U_{g_1f}) = U_{g_2g_1f} \subseteq U_{g_2f} \cap U_{g_1f} \in \mathfrak{F}_x$; \\ (ii) $T_{1}(U_f) = U_f$, for any $U_f \in \mathfrak{F}_x$.\\
Hence, the map $g \mapsto T_g$, where $g \in G(1), U_f \in \mathfrak{F}_x$, defines a representation of $G(1)$ on the group of automorphisms $Aut(\mathfrak{F}_x)$ of the generating set at each maximal ideal $\mathfrak{m}_x$ constituting the Borel $\sigma$-algebra $\mathfrak{B}$.
\end{proof}
\begin{coro}
The family of open sets $\{U_f : f \in \mathfrak{m}_x, x \in X\}$ with the action of $G(1)$ generates a Borel structure ($\sigma$-algebra) on $X$.
\end{coro}
A family $\mathfrak{B} = \{U_f : f \in \mathfrak{m}_x, x \in X \}$ of subsets of $X$, according to Mackey \cite{Mackey57}, is said to generate a unique Borel structure $\mathcal{B}$ on $X$ if it is the smallest Borel structure $\mathcal{B}$ for $X$ which contains $\mathfrak{B}$. In this case, since $\mathfrak{B}$ is the base of the topology on $X$, the Borel structure $\mathcal{B}$ is associated with the topology on $X$. We will now highlight how the association of these subsets of $X$ to the algebra $C(X)$ define their measurability.

\subsection{$C(X)$-Embedding and Measurability}
Given the weak topology on $X$ and the Borel structure generated by the base of closed (or open) sets associated with this topology, we define a homeomorphism as follows. For each $f \in \mathfrak{m}_x$ and $g \in G(1)$, $Z(f)\cap Z(g) = \emptyset$, hence $Z(|f| + |g|) = \emptyset$, define \[ \tau_g : f \mapsto \tau(f,g) = \frac{|g|}{|f| + |g|}. \] Then $\tau_g$ is a measurable function define by each $g \in G(1)$ on the fibre (or on the local bisection) of the groupoid, defined by $\tau_g(f) = \tau(f,g)$. This functions map the zero sets $Z(f)$ to $\{1\}$ and the cozero sets $X-Z(f)$ to the open set $(0,1)$. Thus, it is a homeomorphism on the regular space $X$, such that $\tau(f,g) : X \to (0, 1] \simeq \mathbb{T}$. Thus, the compact space $X$, with the weak topology, is isomorphic to the torus $\mathbb{T}$.

Corresponding to this homeomorphism is a $G(1)$-action defined on the neighbourhood basis (or on the base of closed sets $Z(X)$) of the weak topology on $X$. By definition, the zero sets are balanced sets of the topological space ($X,C(X)$-topology), which is a topological vector space(tvs). The above homeomorphism, taking open subsets of $X$ to open subsets of $\mathbb{T}$, helps us to understand the Borel structure of $G(1)$ and the notion of $C(X)$-embedding of closed sets of $\R$ which is defined in \cite{Gillman1960} as follows.
\begin{defx}
A subspace $S \subset X$ is said to be $C(X)$-embedded if every function in $C(S)$ can be extended to a function in $C(X)$.
\end{defx}
Since every closed set in a metric space is a zero-set, and disjoint sets are completely separated; if the subset $S$ is closed, then its closed subsets are closed in $X$, and completely separated sets in $S$ have disjoint closures in $X$; it follows by Tietse's Extension Theorem that every closed set in $S$ is $C(X)$-embedded. This rests on the fact that every closed set in $\R$ is $C(X)$-embedded. The following theorem expresses this idea.
\begin{thm}\cite{Gillman1960}
If there exists a homeomorphism (function) in $C(X)$ taking $S$ onto a closed set in $\R$, then $S$ is $C(X)$-embedded in $X$.
\end{thm}
\begin{proof}
Let $\tau$ be the homeomorphism (function) in $C(X)$. Then $\theta = \tau^{-1}|_{\tau(S)}$ is a continuous mapping from $H = \tau(S)$ onto $S$, with $\theta(\tau(s)) = s \in S$. Let $f \in C(S)$ be arbitrary. The composite function $f \circ \theta$ is in $C(X)$. Since $H \subset \R$ is closed, by hypothesis, it is embedded; so, there exists $g \in C(\R)$ that agrees with $f\circ \theta$ on $H$. Then $g \circ \tau \in C(X)$, and for all $s \in S$, we have \[(g\circ \tau)(s) = f(\theta(\tau(s))) = f(s),\] which means that $g\circ \tau$ is an extension of $f$.
\end{proof}
The homeomorphism $\tau$ is a unit in $C(X)$. Thus, the $G(1)$-action is related to the $C(X)$-embedding of closed sets of $\R$, and expresses its Borel property. It also highlights the local convexity of the regular compact space $X$ as a $G(1)$-space. We give this as a corollary.
\begin{coro}
The $G(1)$-action on every completely regular compact space $X$, $C(X)$-embeds closed subsets of $\R$ in $X$ thereby making it locally convex.
\end{coro}
In addition, since $X$ is a compact space (even for pseudocompact space) we have $C(X) = C_b(X)$. When $X$ is not compact the homeomorphism may be unbounded. The following corollaries in \cite{Gillman1960} cover both cases.
\begin{defx}
A topological space $X$ is said to be pseudocompact if its image under any continuous real-valued function is bounded.
\end{defx}
\begin{coro}
$X$ is pseudocompact if and only if it contains no $C(X)$-embedded copy of $\mathbb{N}$.
\end{coro}
\begin{coro}
Let $E \subset X$, suppose that some function $\tau \in C(X)$ is unbounded on $E$. Then $E$ contains a copy of $\mathbb{N}$, $C(X)$-embedded in $X$, on which $\tau \to \infty$.
\end{coro}
The two corollaries show the importance of a \emph{lattice} to the ergodicity of the resulting groupoid. Thus, we require the compactness of $X$, or at least a local compactness, to guarantee that $\tau$ is a homeomorphism $X \to [0,1] \simeq \mathbb{T}$ which $C(X)$-embeds closed subsets $[a,b]$ of $\R$ in $X$. Thus, the general Borel transformations can be modelled on the homeomorphism $\tau$ approximated by the action $(0,1] \times X \to X$, where $((t,x),y)$ is the graph of a continuous transformation $g(t) : x \mapsto y$.
\begin{rem}
We therefore make the following remarks. \\
(1) The action is a presentation of $X \times [0,1]$ as a locally convex topological vector space by the continuous functionals $\tau : C(X) \to [0,1]$ separating its points. \\ (2) With this presentation and the notion of $C(X)$-embedding, we see that $\tau$ defines a norm, and hence, a metric on $X$, which brings us to measures on $X$, and the decomposition of a finite (non-ergodic) Borel measure $\mu$ on $X$ with respect to an ultrafilter $Z[\mathfrak{m}_x]$, and the generating set $\mathfrak{B} = \{U_f : f \in \mathfrak{m}_x, x \in X\}$ of the $\sigma$-algebra $\mathcal{B}$.
\end{rem}

\section{The Generalized Space and Dynamics}
The above action shows that $G(1)$ is a Borel group. The dynamical system defined by the Borel group $G(1)$ is by its ergodic action generalized in the action of the algebra/lattice $C(X)$) on the generalized space $\mathcal{M}(X)$, the space of nonnegative Radon measures on $X$. According to \cite{EinsWard2011}, the time evolution of dynamical systems modelled by measure-preserving actions of integers $\mathbb{Z}$ or real numbers $\R$ which represent passage of time are generalized by measure-preserving actions of \emph{lattices} which are usually "subgroup" of Lie groups.

Each maximal ideal $\mathfrak{m}_x$ characterizes and encodes the symmetries of the measurable functions vanishing at each point of $X$ and on its respective neighbourhoods. The symmetry is represented by the $z$-ultrafilter of zero sets (algebraic varieties of $C(X)$) converging to each $x \in X$. The complements of these algebraic sets constitute the open neighbourhood base of points of $X$. The ultrafilter $\mathcal{F}$ convergence to $x$ has associated nets of measurable functions converging to a fixed point function $f$ defined on $x$. This follows from the correspondence between filters and their derived nets. Given a net $f_\alpha$ of contractions in the complete metric space $X$, as described in \cite{HasselblattKatok}, it follows that $f_\alpha \to f$ such that $f(x) = x$. We will represent all these on the generalize space $\mathcal{M}(X)$ with ergodic action of $C(X)$. But we need to recall the ideas of a generalized subset and an ergodic subgroup.

Mackey's conception of a measure class $C$ as a \emph{generalized subset} was extended to the whole space $\mathcal{M}(X)$ of Radon measures on $X$ which is conceived as the \emph{generalized space} of points or \emph{state space} (see \cite{EinsWard2011}). At the centre of this extension is the focus on (i) measure preserving transformations of the compact metric space $X$, and (ii) the Dirac measures $\delta_x$ as generalized or geometric points isomorphic to points of $X$. (The points can also be said to constitute the discrete subgroup $\Gamma$ by their association to $\mathfrak{m}_x$). The role assigned to the ergodic transformations by Mackey, is to translate along time in such a way as to ensure the invariance of measure or state.

Recall the surjective homomorphism $T_x : C(X) \to [0,1)$ which has the maximal ideal $\mathfrak{m}_x$ as its kernel. It follows that a continuous linear transformation $\varphi$ on $X$, by its relation to $C(X)$ as shown above, induces another continuous transformation on the generalized space $\varphi_* : \mathcal{M}(X) \to \mathcal{M}(X)$, and thereby defines a representation of the algebra $C(X)$ on the $\mathcal{M}(X) \simeq L^\infty(X,\mu)$. Notice that the replacement of $\R$ with $[0,1)$ makes all transformations of $X$ contractions.  This translates the whole structure associated to $x \in X$, (the zero sets, the null sets which coincide with the zero sets, and the maximal ideal) from one point to the other.

Thus, the restriction of $\varphi_*$ to the subset $\mathcal{U} = \{\delta_x : x \in X\}$ gives a transformation of $\mathcal{U}$ defined as $\varphi_* : \delta_x \mapsto \delta_{\varphi(x)}$, such that for any $A \subseteq X$, we have \[ (\varphi_*\delta_x)(A) = \delta_x(\varphi^{-1}A) = \delta_{\varphi(x)}(A). \] This restriction shows that the subset $\mathcal{U} = \{\delta_x : x \in X\} \subset \mathcal{M}(X)$ of the generalized points, can be continuously and affinely extended to the generalized space $\mathcal{M}(X)$. Subsequently, if the generalized points $\mathcal{U} = \{\delta_x : x \in X\}$ can generate the generalized space $\mathcal{M}(X)$, then the generalized subsets $\{ C_f : f \in \mathfrak{m}_x\}$ (the measure classes) can be generated from the generalized points, the Dirac measures $\delta_x$. We give this as a result.
\begin{prop}
The generalized space $\mathcal{M}(X)$ of Radon measures on $X$ is an affine and continuous extension of the geometric points $\mathcal{U} := \{\delta_x : x \in X\}$.
\end{prop}
\begin{proof}
The coincidence of zero sets of $C(X)$ with null sets of $\mathcal{M}(X)$ establishes the existence of nets $\{\mu_\alpha\}$ of non-ergodic Radon measures related to $\{\mu_f : f \in \mathfrak{m}_x, x \in X\}$ which converge to the Dirac measures $\{\delta_x : x \in X\}$ as the $z$-ultrafilter $Z[\mathfrak{m}_x]$ converges $x$. Since the elements of $\mathfrak{m}_x$ vanish at $x$, its $G(1)$-action is transferred to fibre of measure classes (or tangent measures to $\delta_x$) via $\varphi_*$ (similar to slice representation in group action).

The transformations $\varphi_* : \mathcal{M}(X) \to \mathcal{M}(X)$ are the natural continuous and affine extensions of the generalized points $\mathcal{U} = \{\delta_x : x \in X\}$ to the generalized space $\mathcal{M}(X)$. This is shown by considering the generalized space $\mathcal{M}(X)$ as generated by the decomposition action of the $C(X)$-modules $\mathfrak{m}_x$, whereby $\mu_f \mapsto \mu_{gf}$ preserves the null set $Z(f)$. This decomposition is identical with the scaling $G(1)$ action which generates $\mathcal{M}(X)$ on the base space $\mathcal{U} = \{\delta_x : x \in X\}$, which is invariant under the $G(1)$-action on $X$. Thus, the Dirac measures $\delta_x \in \mathcal{U}$ constitute the geometric/generalized points of the generalized space $\mathcal{M}(X)$.
\end{proof}
These lead to complementary ways of understanding the concept of ergodicity; which is characterized in different ways for a measure and for a transformation, in \cite{EinsWard2011}. For a measure space $(X,\mathcal{B},\mu)$ with a transformation $\varphi : X \to X$, it means indecomposability of $X$ into two $\varphi$-invariant measurable subsets. But for the transformation $\varphi$ it is defined as follows.
\begin{defx}\cite{EinsWard2011}
A measure preserving transformation $\varphi : X \to X$ of a probability space $(X,\mathcal{B},\mu)$ is ergodic if for any $B \in \mathcal{B}, \varphi^{-1}B = B \; \implies \; \mu(B) = 0$, or $\mu(B) = 1$. Then a $\varphi$-invariant measure $\mu$ is also called ergodic.
\end{defx}
Hence, a non-ergodic measure is infinitely decomposable with respect to a measure-preserving transformation $\varphi$; such that a set of $\varphi$-invariant measures contains the ergodic measures which is $\varphi$-indecomposable as its boundary points. Thus, a closed subset of $\varphi$-invariant measures in $\mathcal{M}(X)$ is either a singleton (when the measure is ergodic) or an infinite set when it contains non-ergodic measures. In the latter, the boundary points are ergodic. Each boundary point makes a complete ergodic meaning with the infinite net converging to it. This (net) completeness defines the nature of the point.

Thus, the dynamism defined by the transformation $\varphi$ is encoded in the symmetry of the measures classes contributing to the convergent net. Hence, the connection between ergodic theory and the dynamics defined by continuous transformations on compact metric spaces, according to \cite{EinsWard2011}, is captured by the closure of the resulting convex set of non-ergodic $\varphi$-invariant measures with ergodic measures as boundary. This also means the convergence of every net in the set. The connection is therefore based on the ergodicity of both measures and transformations. The following is an important result of this section.
\begin{thm}
The action $\varphi$ of the commutative algebra $C(X)$ on the generalized space $\mathcal{M}(X)$ by the Borel group $G(1)$ and the maximal ideals $\mathfrak{m}_x$ has an ergodic limit.
\end{thm}
\begin{proof}
According to \cite{EinsWard2011}, ergodic theorems express a relationship between averages taken along the orbit of a point under the iteration of a measure-preserving map/tansformation $\varphi : X \to X$. The iteration of the transformation on $X \simeq \mathcal{U}$ represents passage of time, and its invariance $\varphi : x \mapsto \varphi(x) \simeq \varphi_*(\delta_x) = \delta_{\varphi(x)}$ represents an average over time.

The induced iteration on the generalized space $\mathcal{M}(X)$ (which is on the fibres) with respect to some invariant measure $\mu$ (or measure class $\varphi_* : C_f \to C_{f\circ \varphi}$) represents the states, which is the net of invariance non-ergodic measures converging to an ergodic limit; the ergodicity represents average over (state) space (averages taken over the classes of measures/states). We have shown that these two averages are defined by invariance of $\mathcal{U}$ and the measure classes $\{C_f\}$ under $G(1)$-actions.
\end{proof}

\begin{rem}
These two averages over \emph{time} and over \emph{states} respectively are of physical \emph{observables} represented by measurable functions. Thus, the replacement of measurable functions with (bounded) operators on a Hilbert space $\mathcal{H}$ leads to quantum mechanical system. In quantum mechanical system, the states of a physical system are represented by vectors in a Hilbert space $\mathcal{H}$. Non-uniqueness of the representation leads to projective spaces in such systems, which relate to integral representations of elements $x \in X$ by a set (convergent sequence or net) of Radon measures.
\end{rem}
In relation to the continuous linear transformations $\varphi : X \to X$, it is also the case that every measure preserving transformation always decomposes into ergodic components. Thus, apart from its connection to dynamics, ergodicity is also connected to the 'noncommutativity' of the quantum system. This is based on the fact that the "noncommutative spaces" replacing the "phase space" are basically quotient spaces determined by ergodic actions of a Borel group. Hence, we have replaced $X$ with the $G(1)$-space $\mathcal{U}$. Within this quotient setting, it is clear that non-ergodic transformations are constituted by ergodic components, which we also consider as limits of nets of the former.

\subsection{General Ergodic Actions of Virtual Groups}
Every algebra has an underlying action of a group, subgroup or (inverse) semigroup. Hence, every algebra can be realized as groupoid algebra. Ergodic actions characterize (measure-preserving) algebras which are generalized by von Neumann algebras. Mackey \cite{Mackey66}, through relaxing the topological structure to a Borel structure, extended the relationship between transitive actions of a group $G$ and its subgroups to a relation between ergodic actions of a Borel group and its virtual subgroups.
\begin{defx}\cite{Mackey66}
A Borel $G$-space is a space $S$ with the action $(s,g) \mapsto sg$ defining a Borel function from $S \times G \to S$.
\end{defx}
Assuming $\mu$ is a measure defined on a $G$-space $S$, that is $\mu : S \to \R_+$. Then $\mu$ is $G$-invariant if $\mu(Bg) = \mu(B) \; \forall \; g \in G$ and all Borel subset $B$ of $S$. The measure $\mu$ is \emph{quasi invariant} if for all $g \in G, \mu(Bg) = 0 \;  \implies \mu(B) = 0$. Since two measures $\mu, \nu$ are in the same measure class $C$ if $N(\mu) = N(\nu)$, a measure class $C$ is invariant if the measure $\nu$ defined by $\nu(B) = \mu(Bg)$ is in the class $C$ whenever $g \in G$ and $\mu \in C$.
\begin{prop}\cite{Mackey66}
The coset space $S_H = G/H$ admits a unique invariant measure class whenever $H$ is closed subgroup of $G$; but invariant measure only for special choices of $H$.
\end{prop}
From the foregoing we characterize a virtual (sub)group. The first characteristic of a virtual group $\mathcal{G}$ is that its $\mathcal{G}$-space has only invariant measure class, which is a generalization of invariant measure. Let $G$ be as in above, and $C$ an invariant measure class in the standard $G$-space $S$. Assume non-transitive $G$-action; then $B$ is invariant subset of $S$, and $S-B$ is also invariant subset. Thus, $S = B\oplus (S-B)$ is the direct sum of $G$-spaces. If $B$ and $S-B$ are not of measure zero, we obtain a non trivial invariant measure class in $B$ and $S-B$ by taking any $\mu \in C$ and then taking the class of its restriction to $B$ and $S-B$ respectively. Hence, $(S,C)$ is realized as the direct sum of two invariant subsystems.

The second characteristic of a virtual (sub)group $G$ is that it acts intransitively or transitively on a space of measure zero. Thus, the invariant (sub)system $(S,C)$ is constructed, according to Mackey, in such a way that: (1) Every measurable invariant subset of $S$ is either of measure zero or the complement of a set of measure zero; (2) Every invariant subset of $S$ on which $G$ acts transitively is of measure zero.

The requirement of ergodicity (or metric transitivity) which is a more basic, inclusive, and sophisticated concept, ensures direct sum decomposition of $S$. Mackey defined the concept as follows.
\begin{defx}\cite{Mackey66}
An action of $G$ on $S$ is ergodic (or metrically transitive) with respect to a measure class $C$ if (1) holds; namely, every measurable invariant subset of $S$ is either of measure zero or the complement of a set of measure zero. When (1) and (2) hold, the action is \emph{strictly ergodic}.
\end{defx}
Generally, in ergodic system, when a subset is invariant, its complement is also invariant; while one is of $\mu$ measure zero, the other is $\mu$-measurable. Thus, ergodic actions of $G$ give a generalization of the transitive actions; and it is constructed to have one-to-one correspondence with the virtual subgroups of $G$ as the transitive $G$-actions have natural one-to-one correspondence with the conjugacy classes of closed subgroups of $G$. The relationship between a null set and an invariant measure class $C$ is the principle of a virtual subgroup as constructed by Mackey.

Assuming $H$ is a closed subgroup of $G$. Every virtual subgroup of $H$ is a virtual subgroup of $G$ also. So, every ergodic action of $G$ is related to the ergodic actions of $H$. The relation is given as follows. If $C$ is an invariant measure class in the standard Borel $G$-space $S$. Then $S \times G$ is the auxiliary $G$-space. Since $G$ acts on $S$ through $H$ and $G$ acts on $H$, $S \times G$ is made a $H \times G$-space by the definition $(s,g_1)(h,g_2) = (sh,g_2^{-1}g_1h)$. Let $(S \times G)_H$ denote the equivalence classes defined by $H$-action on $S \times G$; then we have \[\phi : S \times G \to (S \times G)_H, \; (s,g) \mapsto (sh,gh), \; \text{for } h \in H. \] Since $G$-action commutes with $H$-action on $(S \times G)_H$, it maps a $H$-equivalence class onto another. Hence, \[\phi(s,g_1)g_2 = \phi(s,g_2^{-1}g_1) = (sh, g^{-1}_2g_1h) \] is a well defined $G$-action on $(S \times G)_H$.

The set of $H$-equivalence classes has a Borel structure by a pullback of the Borel structure on $S \times G$. Thus, $E \subset (S\times G)_H$ is a Borel subset if $\phi^{-1}(E)$ is a Borel set in $S\times G$. This makes $(S \times G)_H$ a standard Borel space. The measure class $C$ on $S$ and the measure class of Haar measures $C_G$ on $G$ combine to give the common measure class $C \times C_G$ on $S\times G$ invariant under $H \times G$-action. A typical measure on $S\times G$ is of the form $\mu \times \nu$, where $\mu \in C, \nu \in C_G$; and $\bar{C}$ is an invariant measure class for the $(S \times G)_H$. Thus, $\bar{\omega} \in \bar{C}$ implies $\omega \in C \times C_G$ varies over the finite measures in $C \times C_G$ and $\phi_*\omega = \bar{\omega}$. The following result now follows on this construction.
\begin{prop}\cite{Mackey66}
The measure class $\bar{C}$ is invariant under $G$-action and is $G$-ergodic if, and only if, $C$ is $H$-ergodic.
\end{prop}
\begin{rem} From the forgoing, we remark the following. \\
(1) The construction gives a canonical way of associating every ergodic action of a Lie group $G$ to an ergodic action of its closed subgroup $H$; \\ (2) When $S = H/K$ for some closed subgroup $K$ of $H$, making $H$-action transitive; the $G$-action on $(S\times G)_H$ is also transitive and defined by $K \subset G$. \\ (3) The above construction is generalized by looking at the Borel $S$ as a $(\mathcal{G},H)$-equivalence as defined in \cite{MuhRenWil87}, where $\mathcal{G}$ is the transformation groupoid $S \rtimes G$ and $H = \underset{s \in S}\bigcup \mathcal{G}(s,s)$ is a Lie group bundle. In the following we will try to express the symmetries of these formulation using groupoid equivalence.
\end{rem}

\section{The Measure Groupoid and Action}
The foregoing helps us to present the action of the commutative algebra $C(X)$ using the groupoid structure. We have seen that the $C(X)$-action is determined or defined at each point $x \in X$ by the maximal ideal $\mathfrak{m}_x$ and the Borel group $G(1)$. The maximal ideal is a module of the algebra $C(X)$ with an action of the Borel group $G(1)$ at a given point $x \in X$. Hence, as mentioned earlier, there is a trivialization of an action groupoid on $X$ which we will now explore in order to describe the dynamical system on $X$.

According to Einsiedler and Ward \cite{EinsWard2011}, the difficult and interesting aspect of a dynamical system is to understand the orbit of a point. In this case, the two algebraic objects $\mathfrak{m}_x, G(1)$ help us to understand the dynamics associated to the commutative algebra $C(X)$ at each $x \in X$. Through them we also associate the zero sets $Z[\mathfrak{m}_x]$ which form a $z$-ultrafilter containing $X$ to the dynamical system. Thus, $\{Z(f): f \in \mathfrak{m}_x\}$ is a closed cover for $X$. An open cover for $X$ can be constructed from a countable number of $U_f = X - Z(f)$ as we have seen above; such that $(U_f, \phi_f)$ is an open covering for $X$ and each inverse image $\pi^{-1}(U_f)$ is fibrewise homeomorphic to $U \times \mathfrak{m}_x$. These give a system of homeomorphisms $\phi_f : U_f \times \mathfrak{m}_x \to \pi^{-1}(U_f)$ forming the transition functions \[\phi_{\alpha \beta} = \phi_\beta \circ \phi_\alpha^{-1} : (U_\alpha \cap U_\beta) \times \mathfrak{m}_x \to (U_\alpha \cap U_\beta) \times \mathfrak{m}_x.\]

In line with this construction, we also understand the group of automorphisms or transformations $G(1) \subseteq Aut(X,\mu)$ of $X$ as constituting the structure group of the fibre bundle since it defines an action on the fibres $\mathfrak{m}_x$ given as \[Aut(X,\mu) \times \mathfrak{m}_x \to \mathfrak{m}_x, (\phi, f) \mapsto \phi(f) = f \circ \phi^{-1}.\] The action is fibrewise since $Z(f) = Z(\phi(f))$. Thus, the transformations $\phi \in Aut(X,\mu)$ also define the system of homeomorphisms constituting the transition functions. We have the following result.
\begin{thm}
The symmetries of the commutative algebra $C(X)$ give rise to a (Lie) symmetry groupoid.
\end{thm}
\begin{proof}
The above construction expresses a bundle of symmetries represented with an indexed families $\mathcal{Z} = \{\mathfrak{m}_x\}_{x\in X}$ of geometric structures/points or geometries constituting a bundle $\mathcal{Z}$ over $X$, with projection $p : \mathcal{Z} \to X$, such that $\mathfrak{m}_x = p^{-1}(x)$. The 'symmetry' of these geometric (closed) points of the commutative algebra $C(X)$ is expressed by the groupoid $\mathcal{G}(\mathcal{Z})$. Hence, with $(\mathcal{Z}, p, X)$ a vector bundle, and $\mathcal{G}(\mathcal{Z})$ the set of all vector space isomorphims $\xi : \mathfrak{m}_x \to \mathfrak{m}_y$ for $x,y \in X$. The Borel group $G(1) = G(\mathfrak{m}_x)$ of automorphisms of $\mathfrak{m}_x$ expresses the particular 'symmetry' of $\mathfrak{m}_x$; and the groupoid $\mathcal{G}(\mathcal{Z})$ expresses the smoothly 'varying symmetries' of the bundle.

Thus, in line with Mackenzie \cite{Mackenzie2005}, the smooth bundle symmetry $\mathcal{G}(\mathcal{Z})$ is a Lie groupoid on $X$ with respect to the following structure. For $\xi : G(\mathfrak{m}_x) \to G(\mathfrak{m}_y), s(\xi) = x, t(\xi) = y$; the objection map is $x \mapsto 1_x = Id_{\mathfrak{m}_x}$, the partial multiplication is the composition of maps; the inverse of $\xi \in G(\mathcal{Z})$ is its inverse as an isomorphism. The isotropy groups are the general linear groups $G(\mathfrak{m}_x)$ of the fibres which are all isomorphic.
\end{proof}
\begin{rem}
Given the Lie groupoid $\mathcal{G}(\mathcal{Z}) \rightrightarrows X$. We want to model the Lie groupoid on the generalized space $\mathcal{M}(X)$, which is the representation space of the algebra $C(X)$. We will achieve this by the formulation of the action of the Lie groupoid $\mathcal{G}(\mathcal{Z}) \rightrightarrows X$ on the space of the generalized points $\mathcal{X} = \{\delta_x : x \in X\}$ homeomorphic to $\mathcal{Z}$. This presents the generalized space $\mathcal{M}(X)$ as a measure groupoid giving a generalized measure-theoretic approach to the dynamical system defined by the action of the commutative algebra/lattice $C(X)$ on $X$. This follows from Peter Hahn \cite{Hahn78}.
\end{rem}
From Mackey's definition of generalized subset, we have seen a correspondence between closed subsets of $X$ and Borel (Radon) measures in $\mathcal{M}(X)$; such that the points of $X$ coincide with the Dirac measures $\delta_x \in \mathcal{M}(X)$ (which are also (invariant) ergodic measures in \cite{EinsWard2011}.) The Dirac measures define the \emph{point functionals}.
\begin{defx}\cite{Semadeni65}
A functional $I_x$ on $C(X)$ is called a point functional if $x$ is a point in $X$ such that $I_x(f) = f(x)$ for all $f \in C(X)$.
\end{defx}
Note that $\delta_x(X) = 1, \; \forall \; x \in X$, which makes $\delta_x$ a probability measure. Hence, $\delta_x(f) = f(x)$ for any Borel $A \subset X$ and $f \in C(X)$. We given the action of the Lie groupoid on the set of generalized points $\mathcal{X} = \{\delta_x : x \in X\}$ as a result.
\begin{prop}
Given the Lie groupoid $\mathcal{G}(\mathcal{Z}) \rightrightarrows X$, the set of generalized points $\mathcal{X} =  \{\delta_x : x \in X\}$ is a $\mathcal{G}(\mathcal{Z})$-space.
\end{prop}
\begin{proof}
The homeomorphism $\rho : \mathcal{X} \to X$, which is a continuous open map from the space $\mathcal{X}$ onto the unit space $X$, defines a left action of $\mathcal{G}(\mathcal{Z})$ on $\mathcal{X}$, where $\mathcal{G}(\mathcal{Z}) \star \mathcal{X}$ is the set of composable pair $(\xi,\delta_x)$. This means $(\xi,\delta_x) \mapsto \xi x$ with $s(\xi) = \rho(\delta_x)$. In other words, $\mathcal{X}$ is a left $\mathcal{G}(\mathcal{Z})$-space if $\rho : \mathcal{X} \ni \delta_x \mapsto s(\xi) \in X$. The action defines an equivalence relation on $\mathcal{X}$. Given any pair $\delta_x,\delta_y \in \mathcal{X}$, we say that $\delta_x \sim \delta_y$ if $\rho(\delta_x) = \rho(\delta_y)$ which implies $s(\xi) = \rho(\delta_x) =  \rho(\delta_y)$. Since $\xi$ in $\mathcal{G}(\mathcal{Z})$ are isomorphisms,  $(\xi,\delta_x), (\xi,\delta_y)$ are composable pairs in $\mathcal{G}(\mathcal{Z}) \star \mathcal{X}$. (Cf. \cite{MuhRenWil87})

This action of the Lie groupoid $\mathcal{G}(\mathcal{Z})$ on $\mathcal{X}$ is free and proper. Because $\xi \cdot \delta_x = \delta_x$ implies that $\xi$ is a unit. It is \emph{proper} also because the map $\mathcal{G}(\mathcal{Z}) \star \mathcal{X} \to \mathcal{X} \times \mathcal{X}$ given by $(\xi,\delta_x) \mapsto (\xi \cdot \delta_x, \delta_x)$ is a proper map; that is, the inverse image of a compact set is compact. The two make $\mathcal{X}$ a principal $\mathcal{G}(\mathcal{Z})$-space. Hence, the natural projection $\pi : \mathcal{X} \to \mathcal{G}(\mathcal{Z}) \setminus \mathcal{X}$ onto the locally compact and Hausdorff orbit space $\mathcal{G}(\mathcal{Z}) \setminus \mathcal{X}$ is an open map.

Given that $\mathcal{X}$ is a left principal $\mathcal{G}(\mathcal{Z})$-space, then $\mathcal{X} \star \mathcal{X} = \{(\delta_x,\delta_y) \in \mathcal{X} \times \mathcal{X} : \rho(\delta_x) = \rho(\delta_y) \} \subset \mathcal{X} \times \mathcal{X}$ is the equivalence relation defined by the open map $\rho$ (or $\mathcal{G}(\mathcal{Z})$-action) on $\mathcal{X}$. The equivalence classes are defined by having the same image in $X = \mathcal{G}(\mathcal{Z})^o$. Since $\mathcal{G}(\mathcal{Z})$ acts by composition on $\mathcal{X}$, we have $\xi \cdot \delta_x \; \implies \;  s(\xi) = \rho(\delta_x)$; that is, $\xi \circ \rho$ is defined on $\mathcal{X}$. Thus $\mathcal{X} \star \mathcal{X} \subset \mathcal{X} \times \mathcal{X}$ is a space of equivalence classes or pairs in $\mathcal{X}$ on which a diagonal action of $\mathcal{G}(\mathcal{Z})$ is defined as follows:
\[\mathcal{G}(\mathcal{Z}) \star (\mathcal{X} \star \mathcal{X}) \to \mathcal{X} \star \mathcal{X}, \; \xi \cdot (\delta_x,\delta_y) \mapsto (\xi \cdot \delta_x, \xi \cdot \delta_y). \] Let $H = \mathcal{G}(\mathcal{Z})\setminus \mathcal{X} \star \mathcal{X}$ be the orbit space of the diagonal action. Then $H$ has a natural groupoid structure with multiplication defined as $[\delta_x,\delta_y] \cdot [\delta_y,\delta_z] = [\delta_x,\delta_z]$ with $\mathcal{G}(\mathcal{Z})\setminus \mathcal{X}$ as the unit space. Thus, $\mathcal{G}(\mathcal{Z})\setminus \mathcal{X} \star \mathcal{X} \rightrightarrows \mathcal{G}(\mathcal{Z})\setminus \mathcal{X} = H \rightrightarrows H^o$ is a groupoid, where $t([\delta_x,\delta_y]) = [\delta_x]$ and $s([\delta_x,\delta_y]) = [\delta_y]$, for $[\delta_x], [\delta_y] \in \mathcal{G}(\mathcal{Z}) \setminus \mathcal{X}$.
\end{proof}
\begin{prop}
The groupoid of equivalence defined by $\mathcal{G}(\mathcal{Z})$-action on $\mathcal{X}$ denoted $G(\mathcal{Z})\setminus \mathcal{X} \star \mathcal{X} \rightrightarrows G(\mathcal{Z})\setminus \mathcal{X} = H \rightrightarrows H^o$ defines a right action on the compact space $\mathcal{X}$.
\end{prop}
\begin{proof}
Given the derived groupoid $H \rightrightarrows H^o$, where $\sigma : \mathcal{X} \to H^o$ is a continuous open map from the (locally) compact space $\mathcal{X}$ onto the unit space $H^o = G(\mathcal{Z})\setminus \mathcal{X}$, given as $\delta_x \mapsto [\delta_x] \; \implies \; \sigma(\delta_x) = t([\delta_x,\delta_y]) = t(\xi(\delta_x), \xi(\delta_y)$). Thus, the quotient groupoid $H$ defines a right action on $\mathcal{X}$. We therefore have: \[\mathcal{X} \star H = \{(\delta_z, h) = (\delta_z, [\delta_x,\delta_y]) \in \mathcal{X} \times H : \sigma(\delta_z) = [\delta_z] = [\delta_x] = t([\delta_x,\delta_y]) \}.\] Thus, the action is given by composition $ \delta_z \cdot [\delta_x,\delta_y] = \xi(\delta_y)$, where $g$ is unique in $G(\mathcal{Z})$ and satisfies $\delta_z = g\delta_x$ (because of composability criterion).

The action is well defined for given $[\delta_{x'}, \delta_{y'}] = [\delta_x, \delta_y]$, then there exists a unique $h \in H$ such that $\delta_xh = \delta_{x'}$ and $\delta_yh = \delta_{y'}$. Hence, by definition $[\delta_{x'}] = [\delta_z] \; \implies \; [\delta_x] = [\delta_z]$, and if the three are same orbit then there must be a unique element of $\mathcal{G}(\mathcal{Z})$ such that $\delta_{x'} \mapsto \delta_z$. This is given by $\delta_{x'} \overset{h^{-1}}\mapsto \delta_x \overset{\xi}\mapsto \delta_z$. It therefore follows that \[\delta_z\cdot
[\delta_{x'}, \delta_{y'}] = (\xi h^{-1})\delta_{y'} = \xi(\delta_y) = \delta_z\cdot [\delta_x,\delta_y].\]
\end{proof}
The following diagram illustrates the commutativity of left $\mathcal{G}(\mathcal{Z})$-action and right $H$-action on $\mathcal{X}$.
\begin{center}
\begin{tikzpicture}
\draw [-][black] (0,-1.8,0) node[below] {$\mathcal{G}(\mathcal{Z})$};
\draw [-][black] (3.3,-1.8,0) node[below] {$\mathcal{X}$};
\draw [-][black] (1.5,-2.1,0) node[above] {$*$};
\draw [->][black] (0.3,-2.2,0) -- (2.7,-2.2,0) ;

\draw [->][black] (3,-2.4,0) --  (0.3,-3.8,0);
\draw [->][black] (1.5,-3.1,0) node[above] {$\rho$};
\draw [-][black] (1.5,-4.1,0) node[above] {$\bar{\rho}$};
\draw [->][black] (2.7,-4.2,0) -- (0.3,-4.2,0) ;
\draw [->][black] (0,-2.4,0) --  (0,-3.8,0);
\draw [->][black] (3.3,-2.4,0) -- (3.3,-3.8,0) ;
\draw [-][black] (3.3,-2.9,0) node[right] {$\sigma$};
\draw [-][black] (-0.2,-2.9,0) node[left] {$s$};
\draw [-][black] (0,-3.8,0) node[below] {$X$};
\draw [-][black] (3.3,-3.8,0) node[below] {$\mathcal{G}(\mathcal{Z})\setminus \mathcal{X}$};

\draw [-][black] (5.5,-4.1,0) node[above] {$t$};
\draw [->][black] (6.5,-4.2,0) -- (4,-4.2,0) ;
\draw [-][black] (5.5,-2.1,0) node[above] {$*$};
\draw [->][black] (4,-2.2,0) -- (6.5,-2.2,0) ;
\draw [-][black] (6.9,-1.8,0) node[below] {$\mathcal{X} \star \mathcal{X}$};
\draw [-][black] (6.9,-3.8,0) node[below] {$H$};
\draw [->][black] (6.7,-2.4,0) -- (6.7,-3.8,0) ;
\draw [-][black] (6.7,-2.9,0) node[right] {$\pi_2$};
\end{tikzpicture}
\end{center}
\begin{rem}
This action makes $\mathcal{X}$ a right principal $H$-space, and the left action $\rho$ and the right action $\sigma$ commute $\sigma \circ \rho = \rho \circ \sigma$. The action $\rho$ induces a homeomorphism of $ \bar{\rho} : \mathcal{X}/H \to X = \mathcal{G}(\mathcal{Z})^o$ given as $\bar{\rho}(\delta_z\cdot [\delta_x,\delta_y]) = \rho(\delta_y) = s(\xi)$. The groupoid equivalence of the set of generalized points is now given as a result.
\end{rem}
\begin{prop}
The space of generalized or geometric points $\mathcal{X}$ is a $(\mathcal{G}(\mathcal{Z}),H)$-equivalence.
\end{prop}
\begin{proof}
The proof follows from the above. Since, as we have seen $\mathcal{G}(\mathcal{Z})$ and $H$ are locally compact groupoids, and $\mathcal{X}$ is a (locally) compact space that is (i) a left principal $\mathcal{G}(\mathcal{Z})$-space, (ii) a right principal $H$-space; and (iii) the two actions commute; (iv) the map $\rho : \mathcal{G}(\mathcal{Z}) \star \mathcal{X} \to \mathcal{X}$ induces a bijection of $\mathcal{X}/H$ onto $X$, and (v) the map $\sigma : \mathcal{X} \star H \to \mathcal{X}$ induces a bijection of $\mathcal{G}(\mathcal{Z})\setminus \mathcal{X}$ onto $H^o$.
\end{proof}
\begin{rem}
Because $\mathcal{X}$ is a $(\mathcal{G}(\mathcal{Z}),H)$-equivalence, then $H$ is naturally isomorphic to $\mathcal{G}(\mathcal{Z})\setminus \mathcal{X} \star \mathcal{X}$ and $\mathcal{G}(\mathcal{Z})$ is naturally isomorphic to $\mathcal{X} \star \mathcal{X}/H$. Therefore, if we have $[\delta_x,\delta_y] \in \mathcal{G}(\mathcal{Z})\setminus \mathcal{X} \star \mathcal{X}$, where $\rho(\delta_x) = \rho(\delta_y)$, then there exists a unique $h \in H$ such that $\delta_xh = \delta_y$; the correspondence $[\delta_x,\delta_y] \mapsto h$ is the desired isomorphism between $\mathcal{G}(\mathcal{Z})\setminus \mathcal{X} \star \mathcal{X}$ and $H$.
\end{rem}

An alternative formulation of the above principal groupoid action which is tailored towards pseudogroup is given by Ramsay \cite{Ramsay71}. Another formulation of the above is as the gauge groupoid of a principal $G$-bundle constructed by Mackenzie \cite{Mackenzie2005}. We have shown in \cite{OkekeEgwe2018} that every action of $\mathcal{G}$ on the arrows induces an action on the space of objects. Subsequently, the partial multiplication defines a self-action of the arrows which is reflected on the objects and transferred to $Inj(X)$ of a set $X$ by homomorphisms. The $Inj(X)$ is a pseudogroup which is connected to a net of Radon measures converging to a ergodic measure. For Mackey and Ramsay, the background motivation is to extend the correspondence (similarity) between the homomorphism between groups and their subgroups to ergodic or measure groupoid and the invariant measure classes.

\subsection{Haar System of Measures}
Just as groups act transitively through their closed subgroups, ergodic or measured groupoids act ergodically (or metric transitively) through the closed or invariant measure classes. The class $[\mu]$ of a measure $\mu$ is the set of all equivalent measures to $\mu$ that have the same null set. Every measure class contains a probability since any measure can be normalized on its support. (see \cite{Hahn78}). We apply these to the complete measure algebra $(X,\mathcal{B})$ and the subalgebra determined by the maximal ideal $\mathfrak{m}_x$, that is $(X-Z[\mathfrak{m}_x])$ as defined above. We use Theorem 2.1 in \cite{Hahn78} and the associated definitions to understand the nature of measure classes.

The map $q : X-Z[\mathfrak{m}_x] \to \mathfrak{m}_x$ defined by $U_f \mapsto f$ is a Borel surjection for it corresponds to the partitioning of $X$ by $\mathfrak{m}_x$. Thus, given two equivalent probability measures $\mu \sim \nu$ on $X$. First, their R-N derivative is a positive Borel function $F = \frac{d\mu}{d\nu}$; Second, they give rise to the measures $\bar{\mu} = q_*\mu$ and $\bar{\nu} = q_*\nu$ which are probability measures on $\mathfrak{m}_x \subset C(X)$. Since $\mathcal{Z} = \{\mathfrak{m}_x : x \in X\}$ is a bundle space over $X$, they are defined on the fibres of the groupoid $\mathcal{G}(\mathcal{Z})$, and therefore can be used for the Haar system of measures; Third, the  $\mathfrak{m}_x$-decomposition of $\nu$ is given by the map $(f,\nu) \mapsto \nu_f : \mathfrak{m}_x \times \mathcal{P}(X) \to \mathcal{C}$, where $\mathcal{P}(X)$ is the set of probability measures on $X$, and $\mathcal{C}$ is collection of measure classes satisfying the following: \\
(1) If $h \geq 0$ is  Borel on $U_f$, then $\displaystyle f \mapsto \int hd\nu_f$ is a real-valued Borel function extending it to $X$. \\
(2) $\nu_f(Z(f)) = 0, \; \forall \; f \in \mathfrak{m}_x$ since $Z(f)$ is $\nu_f$-null set.\\
(3) If $h \geq 0$ is Borel on $X$, then $\displaystyle \int hd\nu = \int \left(\int hd\nu_f\right)d\bar{\mu}(f)$.
\begin{rem}
We make the following remarks from the above. \\ (1) Since the decomposition map $(f,\nu) \mapsto \nu_f$ is determined by the zero-sets $Z[\mathfrak{m}_x]$, for any $A \in \mathfrak{F}_x$-the subalgebra, the measure $\displaystyle \bar{\mu}(A) = \int \int \chi_A \circ qd\mu_fd\bar{\mu}(f)$ is defined on the Lie groupoid $\mathcal{G}(\mathcal{Z})$. \\  (2) Almost all $\mu_f$ are probability measures and the R-N property is invariant of the decomposition; thus, $F = \frac{d\mu}{d\nu} = \frac{d\nu_f}{d\mu_f}$ is a Borel function a.e. \\
(3) The equation $\displaystyle \mu = \int \mu_fd\bar{\mu}(f)$ follows on the $\mathfrak{m}_x$-decomposition of $\mu$ (or $q$-decomposition of the probability measure $\mu$ with respect to $\bar{\mu}$ on $\mathfrak{m}_x$). \\ (4) By the equivalence $\mu \sim \nu$, the measure classes $[\nu_f]$, $[\bar{\mu}]$ are determined by the measure class $[\nu]$ or $[\mu]$ up to the zero sets $Z[\mathfrak{m}_x]$.
\end{rem}
The principal Lie groupoid $\mathcal{G}(\mathcal{Z})$ is analytic given that its Borel structure is analytic and the space $(X,\mathcal{B}(X))$ is countably separated. Given a probability measure $\nu$ on the $t$-fibre $t^{-1}(x) = \mathcal{G}(\mathcal{Z})(x,-), x \in X$, an arrow $\xi \in \mathcal{G}(\mathcal{Z})$ with $s(\xi) = x$, and $B \subset \mathcal{G}(\mathcal{Z})$, the map \[ B \mapsto \int \chi_B(\xi\eta)d\nu(\eta) \] defines a probability $\xi \cdot \nu$ on $\mathcal{G}(\mathcal{Z})(t(\xi),-)$. Since the product $\xi\eta$ is defined for $\nu$-almost all $\eta$, the support of $\nu$ is $\mathcal{G}(\mathcal{Z})(s(\xi),-)$. Thus, $supp(\nu^{s(\xi)}) = \mathcal{G}(\mathcal{Z})(s(\xi),-) \; \implies \; supp(\xi\cdot \nu^x) = \mathcal{G}(\mathcal{Z})(t(\xi),-)$, which gives $\xi\cdot \nu^{s(\xi)} = \nu^{t(\xi)}$.

Given a measure $\nu$ on $\mathcal{G}(\mathcal{Z})$, its inverse $\nu^{-1}$ is defined as $\nu^{-1}(B) = \nu(B^{-1})$ for a Borel subset $B$ of $\mathcal{G}(\mathcal{Z})$. Hence, a symmetric measure is defined as follows.
\begin{defx}\cite{Hahn78}
A measure $\mu$ is symmetric if $\mu^{-1} = \mu$. A measure class containing a symmetric measure is symmetric. In that case a measure class $C$ is symmetric if there is $\mu \in C$ such that $\mu \sim \mu^{-1}$.
\end{defx}
\begin{defx}\cite{Hahn78}
Given a symmetric measure class $C$ on the analytic groupoid $\mathcal{G}(\mathcal{Z})$. Let $\mu \in C$ be a probability measure with $t$-decomposition; that is, $\displaystyle \mu = \int \mu^xd\bar{\mu}(x)$ over $X$. Then $\mu$ is called (left) \emph{quasi-invariant} if there is a $\bar{\mu}$-conull Borel set $U_f \subset X$ such that $t(\xi) \in U_f$ and $s(\xi) \in U_f$. Then $\xi\cdot \mu^{s(\xi)} \sim \mu^{t(\xi)}$.
\end{defx}
\begin{defx}(cf. \cite{Hahn78})
A symmetric measure class $C$ possessing a quasi-invariant probability is called \emph{invariant measure class}. With $C$ an invariant measure class, then the pair $(G(\mathcal{Z}), C)$ is a \emph{measure groupoid}.
\end{defx}
\begin{prop}
The principal groupoid $H\rightrightarrows H^o$ defined above is the equivalence relation established on the set of probability measures $\mathcal{P}(X)$ (or on the normalized generalized space $\mathcal{M}_1(X)$) by the action of the Lie groupoid $\mathcal{G}(\mathcal{Z})$.
\end{prop}
\begin{proof}
Two probabilities $(\mu, \nu) \in \mathcal{P}(X) \times \mathcal{P}(X)$ for which there exists $\xi \in \mathcal{G}(\mathcal{Z}) : t(\xi) = \mu, s(\xi) = \nu$, are said to be equivalent. This is an equivalence relation on the $\mathcal{G}(\mathcal{Z})$-space $\mathcal{X}$ defined by the image $\xi \mapsto (t(\xi), s(\xi))$. Therefore, if $\mu, \nu \in \mathcal{X}/{\mathcal{G}(\mathcal{Z})}$, then $\mu \sim \nu$ if on the fibre $\mu_f \sim \nu_f \; \implies \; \xi \mapsto (\mu_f,\nu_f)$. This is in agreement with our understanding of the action of group of invertible elements $G(1) \subset C(X)$ to be transformations on each measure class in $\mathcal{M}(X)$; thus, isomorphisms on the fibres constituted by invariant measure classes given as follows \[\xi \cdot \mu_f^{s(\xi)} \to \nu_{fg}^{t(\xi)}, \; \text{where } \mu_f^{s(\xi)} \sim \nu_{fg}^{t(\xi)}.\]
Hence, the $(H,\mathcal{G}(\mathcal{Z}))$-equivalence of the generalized points $\mathcal{X}$ is related to the action of the commutative algebra (or lattice) $C(X)$ on the generalized space $\mathcal{M}(X)$ and represented by the measure groupoid $(\mathcal{G}(\mathcal{Z}), C)$.
\end{proof}
There is always a symmetric quasi-invariant probability in a measure class. For according to Hahn \cite{Hahn78}, Theorem 2.1, every probability measure in an invariant measure class $C$ is quasi-invariant. Using his conditions, we strengthened the quasi-invariant condition for a Haar measure by modifying it to agree with the maximal ideal $\mathfrak{m}_x$ or the corresponding zero-sets $Z[\mathfrak{m}_x]$ decomposition of measures on $X$ as follows.
\begin{lem}
Let $(\mathcal{G}(\mathcal{Z}), C)$ be a measure groupoid, $\nu \in C$ a probability with $t$-decomposition $\displaystyle \nu = \int \nu^xd\bar{\nu}(x)$. There is a $\bar{\mu}$-conull Borel set $U_f \subset X$ such that \\ (1) $\nu^x(\mathcal{G}(\mathcal{Z})) = 1$ if $x \in U_f$. \\ (2) $\nu^x(\mathcal{G}(\mathcal{Z}) - \mathcal{G}(\mathcal{Z})|_{U_f}) = 0$ if $x \in U_f$. \\ (3) $x \in U_f \; \implies \; \nu^x(t^{-1}(x)) = 1$. \\ (4) if $\xi \in \mathcal{G}(\mathcal{Z})|_{U_f}$, then $\xi\cdot \nu^{s(\xi)} \sim \nu^{t(\xi)}$.
\end{lem}
Given this modification, we now have that for every co-null Borel set $U_f \subset X$, $(\mathcal{G}(\mathcal{Z})|_{U_f}, C)$ is a measure groupoid called an \emph{inessential reduction (i.r)} of $(\mathcal{G}(\mathcal{Z}),C)$ by Hahn \cite{Hahn78}. The inessential reduction for an open conull subset $U_f \subset X$ is also denoted as $\mathcal{G}(\mathcal{Z})_o$.

From this, we see that each invariant measure class $C$ form a system of Haar measures $\{\mu_f^x\}_{x \in X}$ for each $f \in \mathfrak{m}_x$ on the \emph{inessential reduction (i.r)} $(\mathcal{G}(\mathcal{Z})|_{U_f}, C)$. Since each system is defined on the $t$-fibre $t^{-1}(x) = \mathcal{G}(\mathcal{Z})(x,-)$, it follows that the system of Haar measures is not unique. Hence, any invariant measure class $C$ determines the measure groupoid $(\mathcal{G}(\mathcal{Z}), C)$, with a Haar measure defined as follows.
\begin{defx}
Let $(\mathcal{G}(\mathcal{Z}), C)$ be a measure groupoid. Let $\nu \in C$ and let $\mu \in \bar{C} = (t,s)(C)$ be a probability on the base space. The pair $(\nu,\mu)$ is called a Haar measure for $(\mathcal{G}(\mathcal{Z}), C)$ if $\nu$ has a $t$-decomposition $\displaystyle \nu = \int \nu^xd\mu(x)$ with respect to $\mu$ such that for some inessential reduction $\mathcal{G}(\mathcal{Z})_o$ of $\mathcal{G}(\mathcal{Z})$, for all $\xi \in \mathcal{G}(\mathcal{Z})_o$ and $F \geq 0$ a Borel function on $\mathcal{G}(\mathcal{Z})$ we have \[\int F(\gamma)d\nu^{t(\xi)}(\gamma) = \int F(\xi\gamma)d\nu^{s(\xi)}(\gamma). \eqno{(1)}\]
\end{defx}
Thus, for a Borel function $F$ on the groupoid $\mathcal{G}(\mathcal{Z})$, given a symmetric probability $\nu \in C$ with $t$-decomposition $\displaystyle \nu = \int \nu^xd\bar{\nu}(x)$, where $U_f$ is conull as in the above; the quasi-invariance of $\nu$ implies \[(F \mapsto \int F(\gamma)d\nu^{t(\xi)}(\gamma)) \sim (F \mapsto \int F(\xi\gamma)d\nu^{s(\xi)}(\gamma). \]
With these constructions, Hahn showed that every measure groupoid has a measure $\nu$ satisfying (1) for $\xi$ in an inessential reduction.

Furthermore, the symmetricity of the probability $\nu$ implies $t_*\nu = s_*\nu$. Hence, the quasi-invariance for right translation uses the $s$-decomposition $\displaystyle \nu = \int \nu_xd\bar{\nu}(x)$. The result is that the left and right invariance of $C$ are equivalent. Using these, a $(t,s)$-decomposable measure for the composable space $\mathcal{G}(\mathcal{Z})^{(2)}$ is defined as \[ \nu^{(2)} = \int \nu_x \times \nu^xd\bar{\nu}(x), \; \text{with measure class } [\nu^{(2)}],\] which is usually written $C^{(2)}$ and is dependent on $C$. Since the space of composables of any groupoid has a goupoid structure, the conclusion is that $(\mathcal{G}(\mathcal{Z})^{(2)}, C^{(2)})$ is a measure groupoid. The proofs of the following results follow from \cite{Hahn78}, 3.3; 3.4.
\begin{lem}
If $\mathcal{G}(\mathcal{Z})$ is an analytic (standard) Borel groupoid, $\mathcal{G}(\mathcal{Z})^{(2)} \subset \mathcal{G}(\mathcal{Z}) \times \mathcal{G}(\mathcal{Z})$ is an analytic (standard) Borel groupoid.
\end{lem}
\begin{prop}
If $(\mathcal{G}(\mathcal{Z}),C)$ is an analytic groupoid with invariant measure class, so is $(\mathcal{G}(\mathcal{Z})^{(2)},C^{(2)})$.
\end{prop}
The concept of ergodicity is now delineated and related to the dynamical system of the measure groupoid in the sequel. 

\subsection{Ergodicity and Dynamical System}
The measure groupoid is ergodic if and only if there is a single point $x_o \in X$ such that $X\setminus \{x_o\}$ is null. In other words, ergodicity implies the existence of Dirac probability measures $\{\delta_x : x \in X\}$ defined at each point of $X$. These, as we have seen above, constitute the $\mathcal{G}(\mathcal{Z})$-space.

Another way of stating ergodicity, according to Hahn, is that every Borel function $\phi$ on the base $X$ can be expressed in the form of a positive Borel function $F$ on the arrows $\mathcal{G}(\mathcal{Z})$ given as $\phi \circ t(\xi) = F(\xi, s(\xi))$, where $F$ satisfies $F \circ t^{-1} = F \circ s^{-1}$. This means that the Borel functions on the arrows preserve the equivalence relation the groupoid $\mathcal{G}(\mathcal{Z})$ defines on the base space $X$.

Alternatively, a real-valued Borel function $F$ on the measure groupoid $\mathcal{G}(\mathcal{Z})$ satisfying $F(\xi^{-1}\gamma) = F(\gamma)$, $\mu^{t(\xi)}$-a.e and for $\mu$-almost all $\xi$, corresponds to a Borel function $\phi$ on $X$ such that $F = \phi \circ s$ a.e. Thus, the invariant functions on the equivalence space $X \star X$ (or on the space $X$ with $\mathcal{G}(\mathcal{Z})$-action) are of the form $\phi \circ t$ or $\phi \circ s$ (which are homeomorphisms similar to diffeomorphisms $t\circ \sigma$ defined by bisections $\sigma$ on Lie groupoids.) This is in line with our formulation above; it shows the Borel functions are $\mathcal{G}(\mathcal{Z})$-invariant.

Subsequently, the dynamics is related to the convergence of an ultra-filter $\mathcal{F} \to x$. We can therefore define a net of such positive Borel function $F$ on the arrows $\mathcal{G}(\mathcal{Z})$ given as $F_\alpha(\xi, s(\xi)) \to F(\xi, s(\xi))$; or in the form $\phi_\alpha \circ t(\xi) \to \phi \circ t(\xi)$ which can be considered a local bisection. Since such a net $F_\alpha \to F$ (or in terms of local bisections $\phi_\alpha \to \phi$) corresponds to the ultrafilter, it defines or represents the dynamical system of the commutative algebra $C(X)$. The ergodic measure groupoid is therefore defined as follows.
\begin{defx}(see \cite{Hahn78},2.6)
The measure groupoid $(\mathcal{G}(\mathcal{Z}),C)$ is called ergodic if the only Borel functions $\phi : X \to \R$ satisfying $\displaystyle \int|\phi \circ t - \phi \circ s|d\mu = 0$ are such that $\phi = $ constant $\bar{\mu}$-a.e. Alternatively, $(\mathcal{G}(\mathcal{Z}),C)$ is ergodic if and only if for all $A \in \mathcal{B}(X)$, $\displaystyle \int|I_A\circ t - I_A \circ s|d\bar{\mu} = 0 \; \implies \; \bar{\mu}(A) = 0$ or $\bar{\mu}(X - A) = 0.$
\end{defx}
If then we denote the space of all the Borel function on the measure groupoid $\mathcal{G}(\mathcal{Z})$ with $\mathscr{B}(\mathcal{G}(\mathcal{Z}))$, the convolution product of two Borel function $f,g \in \mathscr{B}(\mathcal{G}(\mathcal{Z}))$ on the space is defined as follows. \[f * g(\xi) = \int_{\xi \gamma = \tau}f(\xi)g(\gamma)d\mu^{s(\xi)}(\gamma) = \int_{\mathcal{G}(\mathcal{Z})(t(\xi),-)}f(\xi)g(\xi^{-1}\gamma)d\mu^{t(\xi)}(\gamma). \]
This follows from the involutive map on any Borel $f \in \mathscr{B}(\mathcal{G}(\mathcal{Z}))$ which is defined as $f^*(\xi) = f(\xi^{-1}) = \overline{f(\xi)}$. Thus, the space $\mathscr{B}(\mathcal{G}(\mathcal{Z}))$ is made into a normed $^*$-algebra, with the norm of $f$ given as the supremum norm \[||f|| = \underset{x \in X}\sup |\phi \circ t| = \underset{x \in X}\sup |\phi \circ s| = \underset{x \in X}\sup |\phi(x)|.\] The representation of this convolution algebra $\mathscr{B}(\mathcal{G}(\mathcal{Z}))$ of the measure groupoid is defined on the Hilbert space $\mathcal{H} = L^2(X,\bar{\mu})$ employing the modular function $\Delta$ which Peter Hahn defined in \cite{Hahn78}. We highlight the definition as follows.

The definition of the modular function $\Delta$ makes use of the inverse map on the arrows $i(\xi,\gamma) = (\xi,\gamma)^{-1} = (\xi\gamma, \gamma^{-1})$, the symmetry of the class $[\mu^{(2)}]$ by which $i_*(\mu^{(2)}) \sim \mu^{(2)}$, and a positive Borel map $\rho$ (a cocycle) satisfying \[ \int F(\xi,\gamma)\rho(\xi,\gamma)d\mu^{(2)}(\xi,\gamma) = \int F(\xi\gamma, \gamma^{-1})d\mu^{(2)}(\xi,\gamma), \] where $\rho : G(\mathcal{Z})^{(2)} \to \R^*_+$ is a homomorphism almost everywhere. For a conull Borel set $\mathcal{G}(\mathcal{Z})_o \subset \mathcal{G}(\mathcal{Z})$, $\xi \in \mathcal{G}(\mathcal{Z})_o$ and $E \in \mathcal{B}(\mathcal{G}(\mathcal{Z}))$, Hahn (Lemma 3.6) shows that $\rho$ satisfies the following \[\int I_E(\xi\gamma)d\mu^{s(\xi)}(\gamma) = \int I_E(\gamma)\rho(\gamma^{-1},\xi)d\mu^{t(\xi)}(\gamma);\] For $\mathcal{G}(\mathcal{Z})_1 \subset \mathcal{G}(\mathcal{Z})_o$ \[\int f(\xi\gamma)d\mu^{s(\xi)}(\gamma) = \int f(\gamma)\rho(\gamma^{-1},\xi)d\mu^{t(\xi)}(\gamma).\]
Thus, as $\rho(\xi,\gamma\tau) = \rho(\xi,\gamma)\rho(\xi\gamma,\tau)$ for $\xi,\xi\gamma,\xi\gamma\tau \in \mathcal{G}(\mathcal{Z})_1$.

With this $\rho$ he defined a homomorphism $\sigma : \mathcal{G}(\mathcal{Z})^{(2)} \to G$ into a group which is analytic $C^{(2)}$-a.e, with a Borel function $k : \mathcal{G}(\mathcal{Z}) \to G$ which is defined $C^{(2)}$-a.a $(\xi,\gamma)$ as $\sigma(\xi,\gamma) = k(\xi)p(\xi\gamma)^{-1}$. So, $\rho(\xi,\gamma) = k(\xi)/k(\xi\gamma)$ $\mu^{(2)}$-a.e defines the map $(\xi,\gamma) \mapsto k(\xi)/k(\xi\gamma)$, which given a conull Borel $\mathcal{G}(\mathcal{Z})_o \subset \mathcal{G}(\mathcal{Z})$, with $\xi \in \mathcal{G}(\mathcal{Z})_o$ we obtain
\[\int f(\xi,\gamma)k(\gamma^{-1})d\mu^{s(\xi)}(\gamma) = \int f(\gamma)k(\gamma^{-1}\xi)k(\gamma^{-1})k(\gamma^{-1}\xi)^{-1}d\mu^{t(\xi)}(\gamma)\] \[ = \int f(\gamma)k(\gamma^{-1})d\mu^{t(\xi)}(\gamma) \eqno{(2)}\] for a Borel map $f \geq 0$.

Next, setting $P(\gamma) = k(\gamma^{-1})$, which is a Borel function on $\mathcal{G}(\mathcal{Z})$ for a $t$-decomposable $\displaystyle \mu = \int \mu^xd\bar{\mu}(x)$, a conull Borel set $U_o \subset X$ such that for every $x \in U_o, \mu^x(\mathcal{G}(\mathcal{Z})) = \mu^x(\mathcal{G}(\mathcal{Z})|_{U_o}) = 1$, $\xi \in \mathcal{G}(\mathcal{Z})|_{U_o}, f \geq 0$ on $\mathcal{G}(\mathcal{Z})$, from (2) we have \[ f(\xi\gamma)P(\gamma)d\mu^{s(\xi)}(\gamma) = \int f(\gamma)P(\gamma)d\mu^{t(\xi)}(\gamma).\]
From these we obtain that the set \[K = \{\xi \in \mathcal{G}(\mathcal{Z}) : \int f(\xi\gamma)P(\gamma)d\mu^{s(\xi)}(\gamma) = \int f(\gamma)P(\gamma)d\mu^{t(\xi)}(\gamma)\] is conull and multiplicative. So $K$ contains an \emph{inessential reduction} by the following theorem.
\begin{thm}(see \cite{Hahn78}, Theorem 3.7)
Let $(\mathcal{G}(\mathcal{Z}),C)$ be a a measure groupoid and $H$ an analytic Borel groupoid. Let $\rho : \mathcal{G}(\mathcal{Z}) \to H$ be an a.e. homomorphism. There is an inessential reduction $\mathcal{G}(\mathcal{Z})_o$ of $\mathcal{G}(\mathcal{Z})$ and a Borel function $\rho_o : \mathcal{G}(\mathcal{Z}) \to H$ such that $\rho_o = \rho$ a.e. and $\rho_o|\mathcal{G}(\mathcal{Z})_o$ is a strict homomorphism. Furthermore, if $K \subset \mathcal{G}(\mathcal{Z})$ is conull and $\{\xi\gamma \in \mathcal{G}(\mathcal{Z}) : (\xi, \gamma) \in \mathcal{G}(\mathcal{Z})^{(2)}\}\cap(K \times K) \subset K$, $K$ contains an inessential reduction.
\end{thm}
From these formulations and Theorem 3.8 \cite{Hahn78}, we define the modular function $\Delta$ as follows.
\begin{defx}
The modular function $\Delta := \gamma \mapsto P(\gamma)/P(\gamma^{-1})$ is $\mu^{(2)}$-a.e. homomorphism.
\end{defx}
The modular function $\Delta$ is now used to define a (unitary) representation of the convolution algebra $\mathscr{B}(\mathcal{G}(\mathcal{Z}))$ of the measure groupoid $(\mathcal{G}(\mathcal{Z}),C)$ on $\mathcal{H} = L^2(X,\bar{\mu})$ in the sequel. Notice that $\bar{\mu}$ is a system of Haar measures supported on the fibres $\{\mathfrak{m}_x, x \in X\}$; but given simply as $\bar{\mu}$ because they are same or equivalent measures. Thus, the Hilbert space $\mathcal{H} = L^2(X,\bar{\mu})$ can be considered a bundle space made up of the fibres $\mathcal{G}(\mathcal{Z})(x,-)$. But because the Borel functions $F$ defined on the arrows are equal to Borel functions $\phi \circ t$ defined on $X$, we put the Hilbert space simply as $\mathcal{H} = L^2(X,\bar{\mu})$.

\subsection{A Unitary Representation of the Measure Groupoid}
Given the convolution algebra $\mathscr{B}(\mathcal{G}(\mathcal{Z}))$ of Borel functions defined on the measure groupoid $(\mathcal{G}(\mathcal{Z}), C)$, we now follow Peter Hann formulation of the unitary representation of the convolution algebra $\mathscr{B}(\mathcal{G}(\mathcal{Z}))$ on the space of bounded operators $B(\mathcal{H})$ on the Hilbert space $\mathcal{H} = L^2(X,\bar{\mu})$. His procedure simplified the definition of von Neumann algebra arising from ergodic action of the group $G(1)$ on a measure space $X$. As we have seen, the simplification is achieved above by considering the system of Haar measures on the principal groupoid $\mathcal{G}(\mathcal{Z}) \rightrightarrows X$, and using them in the definition of the convolution algebras (of Borel functions on measure groupoid $(\mathcal{G}(\mathcal{Z}), C)$.)

We employed Peter Hahn's definition of the convolution algebra of a principal groupoid to define the convolution algebra $\mathscr{B}(\mathcal{G}(\mathcal{Z}))$, this gives a $^*$-algebra that coincides with the von Neumann algebra $B(\mathcal{H})$, where $\mathcal{H} = L^2(X,\bar{\mu})$ and $\bar{\mu}$ a probability measure on $X$. The following result on $^*$-representation of the resulting algebra is central to the paper.
\begin{prop}
The map $T : \mathscr{B}(\mathcal{G}(\mathcal{Z})) \to B(\mathcal{H})$ is a $^*$-representation.
\end{prop}
\begin{proof}
The proof follows Hahn \cite{Hahn78}. For the representation, given $f \in \mathscr{B}(\mathcal{G}(\mathcal{Z}))$ and $u,v \in L^2(X,\bar{\mu}) = \mathcal{H}$; define a homomorphism \[T : \mathscr{B}(\mathcal{G}(\mathcal{Z})) \to B(\mathcal{H}), f \mapsto T_f = \int f(\gamma)P(\gamma)d\mu^{t(\gamma)}(\gamma).\] This defines an operator $T_f : \mathcal{H} \to \mathcal{H}$ by $u \mapsto T_f(u)$ such that \[\langle T_f(u), v \rangle = \int f(\gamma)u(s(\gamma))\overline{v(t(\gamma))}P(\gamma)d\mu(\gamma) \]  \[ = \int\left(\int f(\gamma)u(s(\gamma))P(\gamma)d\mu^{x}(\gamma)\right)\overline{v(x)}d\bar{\mu}(x). \] \[= \int T_f(u(x)) \overline{v(x))}d\bar{\mu}(x) \]
Thus, the map $\displaystyle T_fu = \left(x \mapsto \int f(\gamma)u(s(\gamma))P(\gamma)d\mu^{x}(\gamma) \right)$ is also in $\mathcal{H}$, which makes $T_f \in B(\mathcal{H})$.\\
Likewise, we have the map $\displaystyle T : f^* \mapsto \int f(\gamma^{-1})P(\gamma^{-1})d\mu^{t(\gamma)}(\gamma)$ such that \[\langle T_{f^*}(u),v \rangle = \int f^*(\gamma)u(s(\gamma))\overline{v(t(\gamma))}P(\gamma)d\mu(\gamma) = \int f(\gamma^{-1})u(s(\gamma))\overline{v(t(\gamma))}P(\gamma^{-1})d\mu(\gamma)\] \[ = \int \overline{f(\gamma)v(s(\gamma))}u(t(\gamma))P(\gamma)d\mu(\gamma) \] \[ = \overline{\langle T_fv, u \rangle} = \langle u, T_fv \rangle = \langle T^*_{f}u, v \rangle. \; \text{Thus, } T_{f^*} = T^*_f. \]

Given the convolution product $\displaystyle f\star g(\xi) = \int f(\xi)g(\gamma)d\mu^{t(\xi)}(\gamma)$; its image under $T$ is given as follows.
\[T_{f\star g} = \int \int f(\xi\gamma)g(\gamma^{-1})P(\gamma)d\mu^{s(\xi)}(\gamma)P(\xi)d\mu(\xi) \] \[ = \int \int \int f(\gamma\xi)g(\xi^{-1})P(\xi)P(\gamma)d\mu^{s(\gamma)}(\xi)d\mu^x(\gamma)d\bar{\mu}(x); \; \text{by t-decomposition of } \mu \]
\[ =\int \int \int f(\xi)g(\xi^{-1}\gamma)P(\xi)P(\gamma)d\mu^{t(\gamma)}(\xi)d\mu^x(\gamma)d\bar{\mu}(x); \; \text{by convolution property}\]
\[=\int \int f(\xi)\left(\int g(\xi^{-1}\gamma)P(\gamma)\mu^x(\gamma)\right)P(\xi)d\mu^x(\xi)d\bar{\mu}(x) \]
\[=\int \int f(\xi) \left(\int g(\xi^{-1}\gamma)P(\gamma)d\mu^{t(\xi)}(\gamma)\right)P(\xi)d\mu^x(\xi)d\bar{\mu}(x) \]
\[=\int f(\xi)\left( \int g(\gamma)P(\gamma)d\mu^{s(\xi)}(\gamma)\right)P(\xi)d\mu(\xi); \; \text{by reversal of t-decomposition}\]
\[ = T_f \circ T_g.\]
Finally, from $\displaystyle T_{f^*} = \int f^*(\gamma)P(\gamma)d\mu(\gamma)$ we have
\[T_{g^*\star f^*} = \int \int g^*(\xi\gamma)f^*(\gamma^{-1})P(\gamma)d\mu^{s(\xi)}(\gamma)P(\xi)d\mu(\xi) \]
\[ = \int \int f(\gamma)g(\gamma^{-1}\xi^{-1})P(\gamma)d\mu^{s(\xi)}(\gamma)\Delta(\xi)^{-1}P(\xi)d\mu(\gamma).\]
\[ = \int \int f(\xi\gamma)g(\gamma^{-1})P(\gamma)d\mu^{t(\xi)}(\gamma)\Delta(\xi)^{-1}P(\xi)d\mu(\xi); \; (\text{using } \xi \mapsto \xi^{-1}) \]
\[ = \int \overline{g(\gamma)}\left(\int \overline{f(\gamma^{-1}\xi^{-1})}\Delta(\xi)^{-1}P(\xi)d\mu(\xi)\right)P(\gamma)d\mu^{s(\xi)}(\gamma)\]
\[ = T^*_g \circ T^*_f. \]
\end{proof}
The operator is shown to be an isometry as follows. \[|\langle T_fu, v \rangle| = |\int_E f(\gamma)u(s(\gamma))v(t(\gamma))d\mu(\gamma)| \] \[ \leq \int_E |u(s(\gamma))||v(t(\gamma))|d\mu(\gamma); \; \text{since } ||f|| \leq 1 \] \[ \leq \left(\int \int |u(s(\gamma))|^2 d\mu(\gamma)|v(t(\gamma))|^2 d\mu(\gamma)\right)^{1/2} \] \[ = \left(\int |u(x)|^2d\bar{\mu}(x)\right)^{1/2} \left(\int |v(x)|^2d\bar{v}(x)\right)^{1/2} \] \[ = ||u||_2||v||_2. \]
Thus, the $^*$-representation is a unitary representation since $\langle T_fu, T_gv \rangle \leq \langle u, v \rangle$. This is in conformity with our understanding of the Borel functions on the measure groupoid as probability measures on $X$ or Haar system of measures on the groupoid.
\begin{prop}
The convolution algebra $\mathscr{B}(\mathcal{G}(\mathcal{Z}))$ is a commutative von Neuman algebra.
\end{prop}
\begin{proof}
Using Connes' characterization of commutative von Neumann algebra in \cite{Connes94}, 1.3 as the algebras of operators on Hilbert space that are invariant under a group (or subgroup) of unitary operators, it follows that the convolution algebra $\mathscr{B}(\mathcal{G}(\mathcal{Z}))$ of the Lie groupoid $\mathcal{G}(\mathcal{Z})$ is a commutative von Neumann algebra since it is invariant under $Aut(X,\mu) \subset \mathcal{U}(\mathcal{H})$ as given by the condition of ergodicity on $F \circ t^{-1} = F \circ s^{-1}$ on $F$, which implies invariance under transformation of $X$.

The same proposition also follows if we take a second characterization of a commutative von Neumann algebra by \cite{Connes94} as an involutive algebra of operators that is closed under weak limits. The presence of the ultra-filter $\mathcal{F} \to x$ in every fibre $\mathfrak{m}_x$ of the principal Lie groupoid $\mathcal{G}(\mathcal{Z})$ implies, as we have hinted above, the convergence of nets of Borel functions $f_\alpha \to f$ in the normed $^*$-algebra $\mathscr{B}(\mathcal{G}(\mathcal{Z}))$ or nets of local bisections $\phi_\alpha \to \phi$ of the principal Lie groupoid $\mathcal{G}(\mathcal{Z})$. This is also related to the ergodicity of the measure groupoid $(\mathcal{G}(\mathcal{Z}), C)$.
\end{proof}
\begin{coro}
The dynamical system of the commutative von Neumann algebra $\mathcal{N}(\mathcal{H})$ is defined by the convergence of nets of operators $T_\alpha$ defined by the nets of Borel functions $f_\alpha \to f$ or local bisections $\phi_\alpha \to \phi$.
\end{coro}
\begin{proof}
This follows from the definition of the operators above as \[ f \mapsto T_f = \int f(\gamma)P(\gamma)d\mu^{t(\gamma)}(\gamma).\] That these nets define the dynamical system of the von Neumann algebra follows from their relationship to the convergence of nets of non-ergodic measures to ergodic limits, which as given in \cite{EinsWard2011}, represents the dynamical system of ergodic actions.
\end{proof}
From these results, it follows that the measure groupoid $(\mathcal{G}(\mathcal{Z}),C)$ has invariant system of Haar measures defined by a measure class $C$ is in agreement with the position that the measure classes are the orbits of the action of the algebra $C(X)$ on the generalized space $\mathcal{M}(X)$. The resulting dynamical system is defined or represented on the generalized space by convergent nets of measures $\mu_\alpha \to \delta_x$ with ergodic limits corresponding to $z$-ultrafilters $Z(\mathfrak{m}_x) \to x$ in $X$.

Further, the existence of many measure classes for the $^*$-representation points to the fact that the $^*$-representation of the convolution algebra of the measure groupoid is not uniquely tied to any measure class. In other words, the left Haar system of measures is not unique. The homomorphism of the $^*$-representation implies that the convergence of a net of Borel function on the convolution algebra $\mathscr{B}(\mathcal{G}(\mathcal{Z}))$ implies a net of bounded unitary operators in the von Neumann algebra $\mathcal{N}(\mathcal{H})$.

\section{conclusion}
We have presented the commutative algebra $C(X)$-action on the generalized space $\mathcal{M}(X)$ as an extension of the polar action of the automorphism group $Aut(X,\mu)$, in the form of the action of the measure groupoid $G(\mathcal{Z})$ on the space $\mathcal{X}$ of geometric or closed points of the $\mathcal{M}(X)$. Ergodic requirements made it into the dynamical system defined by the commutative von Neumann algebra $\mathcal{N}(\mathcal{H})$ on the Hilbert space $\mathcal{H} = L^2(X,\bar{\mu})$. Hence, the presentation of the geometric space $\mathcal{X} = \{\delta_x : x \in X\}$ as a $(G(\mathcal{Z}),H)$-equivalence has import for the $^*$-representation of the convolution algebra $\mathscr{B}(\mathcal{G}(\mathcal{Z}))$ of the principal Lie groupoid $\mathcal{G}(\mathcal{Z})$ or the measure groupoid $(G(\mathcal{Z}), \mathcal{C})$ on the von Neumann algebra $B(\mathcal{H})$ of bounded operators on $\mathcal{H}$.

Considered as the action of a compact (Borel) group $G(1) \subset C(X)$, the orbit of the principal Lie groupoid $\mathcal{G}(\mathcal{Z})$ action is an immersed submanifold or section $\Sigma$, which meets every measure class orthogonally; which gave rise to the uniqueness of the measure groupoids $(\mathcal{G}(\mathcal{Z}),C)$. While in the case of a transitive Lie groupoid, the section $\Sigma$ is defined by a horizontal distribution $\mathcal{H}$ whose integral manifold gives the foliation of the groupoid, here it is defined by nets of $Aut(X,\mu)$-invariant Borel measures (or Borel functions) related to the null hyperplane $\mathfrak{m}_x$. Therefore the net define the dynamical system of the von Neummann algebra corresponding to the convolution algebra $\mathscr{B}(\mathcal{G}(\mathcal{Z}))$ of the principal Lie groupoid $\mathcal{G}(\mathcal{Z})$. 

Thus, the classical Chevalley restriction Theorem for the adjoint action of a compact group $G$ applies here also. This states that since the ergodic subgroup $Aut(X,\mu)$ is compact, a $Aut(X,\mu)$-invariant function on $\mathcal{M}(X)$ is integrable if and only if its restriction to the section $\Sigma$ is integrable, and invariant under the stabilizer of the section $\Sigma$; where the section $\Sigma$ in this case is given by the measure groupoid $(\mathcal{G}(\mathcal{Z}), C)$. See also \cite{GroveZiller2012}.

\bibliographystyle{amsplain}

\providecommand{\bysame}{\leavevmode\hbox to3em{\hrulefill}\thinspace}

\end{document}